\documentclass[a4paper]{amsart}

\usepackage[alphabetic,initials,nobysame]{amsrefs}
\usepackage{amssymb,mathtools,mathrsfs,hyperref,tikz}
\usepackage{comment,todonotes,color} 
\usepackage[shortlabels]{enumitem}

\BibSpec{collection.article}{%
	+{}  {\PrintAuthors}				{author}
	+{, } {}                     		{title}
	+{.} { }                            {part}
	+{:} { \textit}                     {subtitle}
	+{,} { \PrintContributions}         {contribution}
	+{,} { \PrintConference}            {conference}
	+{}  {\PrintBook}                   {book}
	+{. In }{\textit}                   {booktitle}
	+{,} { }							{series}
	+{,} { pp.~}                        {pages}
	+{,} { }							{publisher}
	+{,} { }							{address}
	+{,} { \PrintDateB}                 {date}
	+{,} { }                            {status}
	+{,} { \PrintDOI}                   {doi}
	+{,} { available at \eprint}        {eprint}
	+{}  { \parenthesize}               {language}
	+{}  { \PrintTranslation}           {translation}
	+{;} { \PrintReprint}               {reprint}
	+{.} { }                            {note}
	+{.} {}                             {transition}
	+{}  {\SentenceSpace \PrintReviews} {review}
}

\theoremstyle{plain}
\newtheorem{theorem}{Theorem}
\newtheorem{lemma}[theorem]{Lemma}

\newtheorem{prop}[theorem]{Proposition}
\newtheorem{corollary}[theorem]{Corollary}

\theoremstyle{definition}

\theoremstyle{remark}

\numberwithin{equation}{section}
\numberwithin{theorem}{section}

\newcommand{\br}{\overline}
\newcommand{\R}{\mathbb R}

\newcommand{\C}{\mathbb C}

\newcommand{\N}{\mathbb N}

\newcommand{\h}{\mathscr H}

\DeclareMathOperator{\dist}{dist}
\DeclareMathOperator{\diam}{diam}
\DeclareMathOperator{\id}{{\mathrm{id}}}

\DeclareMathOperator{\md}{Mod}

\DeclareMathOperator{\area}{\mathrm{Area}}
\DeclareMathOperator{\loc}{loc}

\renewcommand{\mod}{\mathrm{Mod\,}}

\begin{document}
\title{Metric surfaces and conformally removable sets in the plane}
\author{Dimitrios Ntalampekos}
\address{Department of Mathematics, Aristotle University of Thessaloniki, Thessaloniki, 54152, Greece.}
\thanks{The author was partially supported by NSF Grant DMS-2246485.}
\email{dntalam@math.auth.gr}
\date{\today}

\begin{abstract}
We characterize conformally removable sets in the plane with the aid of the recent developments in the theory of metric surfaces. We prove that a compact set in the plane is $\mathit S$-removable if and only if there exists a quasiconformal map from the plane onto a metric surface that maps the given set to a set of linear measure zero. The statement fails if we consider maps into the plane rather than metric surfaces. Moreover, we prove that a set is $S$-removable (resp.\ $CH$-removable) if and only if every homeomorphism from the plane onto a metric surface (resp.\ reciprocal metric surface) that is quasiconformal in the complement of the given set is quasiconformal everywhere. 
\end{abstract}
\maketitle

\section{Introduction}

Let $E\subset \C$ be a compact set.  We say that $E$ is \textit{$S$-removable} if every conformal embedding $f\colon \C\setminus E\to \C$ is the restriction of a conformal map of the Riemann sphere $\widehat \C$. Ahlfors--Beurling \cite{AhlforsBeurling:Nullsets} studied carefully $S$-removable sets in the plane and provided deep characterizations for them. If $\h^1(E)=0$, then $E$ is $S$-removable, as a consequence of Painlev\'e's theorem; see Theorem 2.7 and Proposition 4.3 in \cite{Younsi:removablesurvey}. Moreover, by the quasiconformal invariance of $S$-removable sets \cite{Younsi:removablesurvey}*{Proposition 4.7}, one immediately derives the following sufficient condition.

\begin{prop}\label{proposition:qc_hausdorff}
Let $E\subset \C$ be a compact set. If there exists a quasiconformal homeomorphism $f\colon \C \to \C$ such that $\h^1(f(E))=0$, then $E$ is $S$-removable. 
\end{prop}

Can this condition provide a characterization of $S$-removable sets? The answer is negative, as the following simple example illustrates. Consider a totally disconnected compact set $E\subset \R$ with Hausdorff dimension $1$ and $1$-dimensional Lebesgue measure zero. Then $E\times E$ has Hausdorff dimension $2$ and is $S$-removable because its projections to the coordinate axes have  $1$-dimensional measure zero \cite{AhlforsBeurling:Nullsets}*{Theorem 10}. On the other hand, sets of Hausdorff dimension $2$ cannot be mapped under quasiconformal maps of the plane to sets of lower dimension \cite{GehringVaisala:dimension}*{Corollary 13}. In particular, $E\times E$ cannot be mapped to a set of Hausdorff $1$-measure zero with a quasiconformal map of $\C$.

We show that the condition of Proposition \ref{proposition:qc_hausdorff} provides a characterization of $S$-removable sets if one \textit{alters quasiconformally the metric} of $\C$. A metric surface $X$ is a metric space that is homeomorphic to $\C$ and has locally finite Hausdorff $2$-measure.

\begin{theorem}\label{theorem:s_0}
A compact set $E\subset \C$ is $S$-removable if and only if there exists a quasiconformal homeomorphism $f$ from $\C$ onto a metric surface $X$ such that $\h^1(f(E))=0$.
\end{theorem}

We also include another characterization involving metric surfaces.

\begin{theorem}\label{theorem:s_def}
A compact set $E\subset \C$ is $S$-removable if and only if every homeomorphism $f$ from $\C$ onto a metric surface $X$ that is quasiconformal on $\C\setminus E$ is quasiconformal on $\C$.
\end{theorem}

Note that both statements are false if we do not allow $X$ to be an arbitrary metric surface, but only consider the case $X=\C$. Both theorems are consequences of Theorem \ref{theorem:ned_hausdorff}. These results expand on the main results of Ikonen--Romney \cite{IkonenRomney:removable} who studied the relation between $S$-removable sets and a specific class of surfaces arising by scaling the Euclidean metric with a conformal weight. 

In recent years metric surfaces have been studied in connection with the problems of quasisymmetric and quasiconformal uniformization. Most prominently, the Bonk--Kleiner theorem \cite{BonkKleiner:quasisphere} provides sufficient conditions for the existence of a quasisymmetric parametrization of a metric sphere and works due to Meier, Rajala, Romney, Wenger, and the current author \cites{Rajala:uniformization, NtalampekosRomney:length, MeierWenger:uniformization,  NtalampekosRomney:nonlength} provide a very satisfactory understanding of the quasiconformal parametrization problem. Not every surface admits a quasiconformal parametrization by the complex plane. Rajala \cite{Rajala:uniformization} introduced the class of \textit{reciprocal} surfaces and proved that these are precisely the metric surfaces that admit a quasiconformal parametrization by a planar domain. 

A class of sets that are related to $S$-removable sets are $CH$-removable sets. A compact set $E\subset \C$ is \textit{$CH$-removable} if every homeomorphism $f\colon \C\to \C$ that is conformal in $\C\setminus E$ is the restriction of a conformal map of $\C$. Equivalently, we may consider quasiconformal maps instead of conformal in this definition \cite{Younsi:removablesurvey}*{Proposition 5.7}. The class of $CH$-removable sets is significantly larger than the class of $S$-removable sets, but we do not have a satisfactory understanding as in the case of $S$-removable sets; see \cite{Ntalampekos:cned} for recent developments in this direction. The following statement is an immediate consequence of the definitions.

\begin{theorem}\label{theorem:ch_def}
A compact set $E\subset \C$ is $CH$-removable if and only if every homeomorphism $f$ from $\C$ onto a reciprocal metric surface $X$ that is quasiconformal on $\C\setminus E$ is quasiconformal on $\C$.
\end{theorem}

Upon comparing Theorems \ref{theorem:s_def} and \ref{theorem:ch_def} we see that the difference between $S$-removable and $CH$-removable sets is closely related to the difference between arbitrary metric surfaces and reciprocal metric surfaces. 

One can ask whether there is a result analogous to Theorem \ref{theorem:s_0} for $CH$-re\-movable sets. We are able to provide a sufficiency criterion.

\begin{theorem}\label{theorem:intro_ch}
Let $E\subset \C$ be a compact set. Suppose that there exists a quasiconformal homeomorphism $f$ from $\C$ onto a metric surface $X$ such that $f(E)$ has $\sigma$-finite Hausdorff $1$-measure. Then $E$ is $CH$-removable. 
\end{theorem}

The proof is given in Section \ref{section:ch}. It would be very interesting to know whether the converse to this statement is true, even for totally disconnected sets $E$.

Let $E\subset \C$. We define the \textit{geometric quasiconformal dimension} of $E$ to be the infimum of Hausdorff dimensions of sets  $f(E)$, over all quasiconformal homeomorphisms $f$ from $\C$ onto metric surfaces $X$; throughout the paper we use the so-called geometric definition of quasiconformality, whence the name geometric quasiconformal dimension. We prove the following results regarding the geometric quasiconformal dimension.

\begin{theorem}\label{theorem:s_gd0}
A compact set $E\subset \C$ is $S$-removable if and only if its geometric quasiconformal dimension is $0$. 
\end{theorem}

The result follows from Theorem \ref{theorem:ned_hausdorff} below. In fact, we show that the dimension is attained. Combining Theorem \ref{theorem:s_gd0} with Theorem \ref{theorem:s_0}, we obtain the next result. 

\begin{corollary}\label{corollary:gd_0_1}
There are no sets in the plane with geometric quasiconformal dimension strictly between $0$ and $1$.
\end{corollary}

A similar result is true for other notions of dimension such as the conformal dimension \cite{Kovalev:conformal_dimension} and the conformal Assouad dimension \cite{Heinonen:metric}*{Section 15.8}. It seems plausible that $CH$-removable sets have geometric quasiconformal dimension at most $1$, which would follow if the converse of Theorem \ref{theorem:intro_ch} were true. We remark that the Sierpi\'nski gasket is an example of a set that is not $CH$-removable \cite{Ntalampekos:gasket}, yet it has geometric quasiconformal dimension equal to $1$ \cite{TysonWu:Gasket}.

An ingredient in the proof of Theorem \ref{theorem:ned_hausdorff} is the next fundamental fact, which we have not been able to locate in the literature and is established in Section \ref{section:disconnected}.

\begin{theorem}\label{theorem:totally_disconnected_ned}
If a compact set $K\subset \C$ is not $S$-removable, then there exists a totally disconnected compact set $K'\subset K$ that is not $S$-removable.
\end{theorem}

The corresponding statement for $CH$-removable sets remains an open problem \cite{Bishop:flexiblecurves}*{Question 4}. 
 
\subsection*{Acknowledgment}
The author would like to thank Malik Younsi for his comments on an earlier draft of the paper.

\section{Totally disconnected sets and S-removability}\label{section:disconnected}

In this section we prove Theorem \ref{theorem:totally_disconnected_ned}. We provide some background first. 

\subsection{Hausdorff measure}
Let $X$ be a metric space. For $s\geq 0$ the \textit{$s$-dimensional Hausdorff measure} $\mathscr H^s(E)$ of a set $E\subset X$ is defined by
$$\mathscr{H}^{s}(E)=\lim_{\delta \to 0} \mathscr{H}_\delta^{s}(E)=\sup_{\delta>0} \mathscr{H}_\delta^{s}(E),$$
where
$$
\mathscr{H}_\delta^{s}(E)=\inf \left\{ c(s)\sum_{j=1}^\infty (\operatorname{diam}(U_j))^s: E \subset \bigcup_j U_j,\, \operatorname{diam}(U_j)<\delta \right\}
$$
for a normalizing constant $c(s)>0$ so that the $n$-dimensional Hausdorff measure agrees with Lebesgue measure in $\R^n$. The Hausdorff dimension of $E$ is defined as 
$$\dim_\h(E)= \inf\{s\geq 0: \h^s(E)=0\}.$$
If $\delta=\infty$, the quantity $\mathscr{H}_\infty^{s}(E)$ is called the \textit{$s$-dimensional Hausdorff content} of $E$ and is an outer measure on subsets of $X$. An elementary fact is that
\begin{align*}
\textrm{$\h^s(E)=0$ if and only if $\h^s_{\infty}(E)=0$. }
\end{align*}
We always have
$$\min \{\h^1(E),\diam(E) \}\geq \h^1_\infty(E)$$
and if $E$ is connected, then 
\begin{align*}
\h^1_{\infty}(E)= \diam(E).
\end{align*}
See \cite{BuragoBuragoIvanov:metric}*{Lemma 2.6.1, p.~53} for an argument.

\subsection{Modulus}

Let $\Gamma$ be a family of curves in a metric surface $X$. A Borel function $\rho\colon X\to [0,\infty]$ is \textit{admissible} for $\Gamma$ if $\int_\gamma \rho\, ds\geq 1$ for all rectifiable paths $\gamma\in \Gamma$. We define the \textit{$2$-modulus} of $\Gamma$ as
$$\mod \Gamma=\inf_\rho \int_X \rho^2 \, d\h^2,$$
where the infimum is taken over all admissible functions $\rho$ for $\Gamma$. In this section we only consider modulus in the plane, where Hausdorff $2$-measure agrees with the Lebesgue measure. We will use the following standard facts about modulus in a space $X$.

\begin{enumerate}[label=(M\arabic*)]
\item The modulus $\mod$ is an outer measure in the space of all curves in $X$. In particular, it obeys the monotonicity and countable subadditivity laws. \label{m:outer_measure}

\smallskip

\item If every path of a family $\Gamma_1$ has a subpath lying in a family $\Gamma_2$, then $\mod \Gamma_1\leq \mod \Gamma_2.$\label{m:subordinate}

\smallskip

\item Modulus is invariant under conformal maps on subsets of the complex plane.

\smallskip

\item For $x\in \C$ and $0<r<R$ let $\Gamma=\Gamma(A(x;r,R))$ be the family of curves in $\C$ joining the boundary components of the annulus $A(x;r,R)=B(x,R)\setminus \br B(x,r)$.  Then 
$$\mod \Gamma= 2\pi \left(\log\frac{R}{r} \right)^{-1}.$$\label{m:ring}
\end{enumerate}
See \cite{Vaisala:quasiconformal}*{Chapter 1} and \cite{HeinonenKoskelaShanmugalingamTyson:Sobolev}*{Sections 5.2--5.3} for more details about modulus and proofs of these facts.

\subsection{Proof of Theorem \ref{theorem:totally_disconnected_ned}}
For sets $E,F,G$ in a metric space $X$ denote by $\Gamma(E,F;G)$ the family of curves $\gamma\colon [a,b]\to X$ such that $\gamma(a)\in E$, $\gamma(b)\in F$, and $\gamma((a,b))\subset G$.  A \textit{quadrilateral} in a surface $X$ is a closed Jordan region $Q$ together with a partition of $\partial Q$ into four non-overlapping edges $\zeta_1,\zeta_2,\zeta_3,\zeta_4\subset \partial Q$ in cyclic order. When we refer to a quadrilateral $Q$, it is implicitly understood that there exists such a marking on its boundary and we write $Q=Q(\zeta_1,\zeta_2,\zeta_3,\zeta_4)$ to indicate the marking. We define $\Gamma(Q)=\Gamma(\zeta_1,\zeta_3;Q)$ and $\Gamma^*(Q)=\Gamma(\zeta_2,\zeta_4;Q)$.

Ahlfors and Beurling characterized $S$-removable sets as sets the removal of which does not affect modulus in the following precise sense.

\begin{theorem}[\cite{AhlforsBeurling:Nullsets}*{Theorem 9}]\label{theorem:ned_removable}
A compact set $K\subset \C$ is $S$-removable if and only if for each quadrilateral $Q\subset \C$ we have
$$\mod \Gamma(Q) =\mod \Gamma(\zeta_1,\zeta_3; Q\setminus K).$$
\end{theorem}

We establish some auxiliary results before presenting the proof of Theorem \ref{theorem:totally_disconnected_ned}. 

\begin{lemma}\label{lemma:totally_disconnected_ned_sufficient}
Let $\Omega \subset {\C}$ be a domain such that $\partial \Omega$ is a non-degenerate continuum and for each $\varepsilon>0$ and each continuum $E\subset \Omega$  there exists an open set $U\subset N_{\varepsilon}(\partial \Omega)$ with the property that $K'=\partial \Omega\setminus U$ is totally disconnected and
$$\mod \Gamma(E, U; \C\setminus K')<\varepsilon.$$
Then $\partial \Omega$ has a totally disconnected compact subset that is not $S$-removable. 
\end{lemma}
\begin{proof}
Consider a rectangle $Q=Q(\zeta_1,\zeta_2,\zeta_3,\zeta_4)$, where $\zeta_1\subset \Omega$, $\zeta_1$ and $\zeta_3$ are disjoint from $\partial \Omega$, and  every curve of $\Gamma(Q)$ intersects $\partial \Omega$; see Figure \ref{figure:rectangleQ}.
\begin{figure}
\begin{tikzpicture}
	\begin{scope}	
	\draw[rounded corners=7] (0.5,-1)--(1.25,-1.4)--(1.5,-0)--(1.75,0.4)--(2,-0.6)--(2.5, 1) -- (2,1.5)--(1,0.5)--(0.5,1)--(0,0)--cycle;
	
	\draw (1,-1.2) rectangle (3,-0.3);
	
	\node[left] at (1.1,-0.7) {$\zeta_1$};
	\node[right] at (3,-0.7) {$\zeta_3$};
	\node at (1,0) {$\Omega$};
	\end{scope}
	
	\begin{scope}[shift={(6,0)}]
	\draw (0,-0.8) rectangle (3,0.8);	
	\draw[rounded corners=7] (0.5,-1)--(1.25,-1.4)--(1.5,-0)--(1.75,0.4)--(2,-0.6)--(2.5, 1) -- (2,1.5);
	
	\node[left] at (0,0) {$\zeta_1$};
	\node[right] at (3,0) {$\zeta_3$};
	\node[left] at (1.5,0) {$\partial \Omega$};
	\end{scope}
\end{tikzpicture}
\caption{The rectangle $Q$ in the cases that $\Omega$ is bounded and unbounded, respectively.}\label{figure:rectangleQ}
\end{figure}
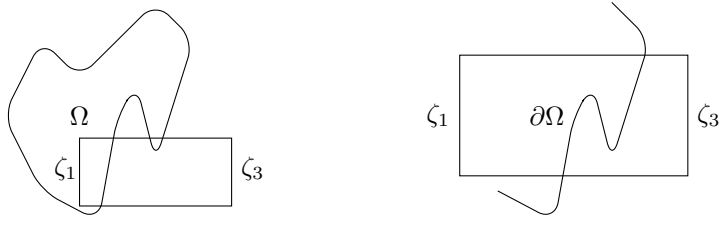
The existence of $Q$ can be justified as follows. If $\Omega$ is bounded then one can simply take a side $\zeta_1$ inside $\Omega$ and a side $\zeta_3$ outside $\Omega$ to construct $Q$. If $\Omega$ is unbounded then one starts with a line that disconnects the boundary of $\Omega$ (this requires the boundary of $\Omega$ to have more than one point). Then we slightly fatten the line to obtain an infinite strip with the property that any (infinite) line in the strip intersects $\partial \Omega$ due to connectedness. One can now take a rectangle in the strip.

Now, let $\varepsilon>0$ such that 
$$\varepsilon< \min \{\dist(\zeta_1\cup \zeta_3,\partial \Omega) , \mod\Gamma(Q)\}.$$
By assumption, there exists an open set $U\subset N_{\varepsilon}(\partial \Omega)$ such that $K'=\partial \Omega\setminus U$ is totally disconnected and 
$$\mod \Gamma(\zeta_1, U; \C\setminus K')<\varepsilon.$$
Every path $\gamma\in \Gamma(\zeta_1,\zeta_3;Q\setminus K')$ has a subpath in $\Gamma(\zeta_1,U;\C\setminus K')$. Therefore, by property \ref{m:subordinate} we have
$$\mod \Gamma(\zeta_1,\zeta_3;Q\setminus K')\leq \mod \Gamma(\zeta_1,U;\C\setminus K')<\varepsilon<\mod \Gamma(Q).$$
With the aid of Theorem \ref{theorem:ned_removable}, this shows that $K'$ is not $S$-removable. 
\end{proof}

\begin{lemma}\label{lemma:countable}
Let $G\subset \C$ be a countable set and $E\subset \C\setminus  G$ be a compact set. Then for each $\varepsilon>0$ there exists an open set $U$ with $G\subset U\subset N_{\varepsilon}(G)$ such that 
$$\mod \Gamma(E,U;\C)<\varepsilon.$$
\end{lemma}
\begin{proof}
By the countable subadditivity of modulus (see \ref{m:outer_measure}), it suffices to show the conclusion assuming that $G$ contains only one point $x_0$. Let $R>0$ such that $B(x_0,R)\cap E=\emptyset$. For $0<r<R$, let $U_r=B(x_0,r)$ and note that each path in $\Gamma(E,U_r;\C)$ has a subpath in $\Gamma(A(x_0;r,R))$. By \ref{m:subordinate} and \ref{m:ring}, we have
$$\mod \Gamma(E,U_r;\C) \leq \mod \Gamma(A(x_0;r,R))= 2\pi \left(\log\frac{R}{r} \right)^{-1}. $$
As $r\to 0$, this converges to $0$, so the conclusion follows. 
\end{proof}

For two distinct parallel lines $L_1,L_2$ in $\C$ denote by $W(L_1,L_2)$ the closed strip that they bound. We will use the following characterization of local connectedness \cite{LoridantLuoYang:plane}*{Theorem 2}; see also \cite{Whyburn:topology}*{Theorem (12.1), p.\ 18} for a similar statement.

\begin{theorem}\label{theorem:characterization_locally_connected}
Let $K\subset \C$ be a continuum. Then $K$ is locally connected if and only if for every pair of distinct parallel lines $L_1,L_2$ the intersection $W(L_1,L_2)\cap K$ has at most finitely many components that intersect both $L_1$ and $L_2$.  
\end{theorem}

We also need the next topological lemma. See Figure \ref{figure:circles} for an illustration of the assumptions.

\begin{lemma}\label{lemma:convergence}
Let $F\subset \C$ be a continuum and $\{F_n\}_{n\in \N}$ be a sequence of continua in $\C$ that converges to $F$ in the Hausdorff sense. For each $n\in \N$ let $V_n$ be a bounded component of $\C\setminus (F\cup F_n)$ that is not contained in a bounded component of $\C\setminus F$. Then for each $\delta>0$ we have $V_n\subset N_\delta(F)$ for all sufficiently large $n\in \N$. 
\end{lemma}
\begin{proof}
We argue by contradiction. By passing to a subsequence, we assume that there exists $\delta>0$ and a sequence $x_n\in V_n$, $n\in \N$, such that $\dist(x_n,F)\geq \delta$ for all $n\in \N$. Note that $\partial V_n\subset F\cup F_n$ and $\dist(x_n,\partial V_n)=\dist(x_n,F\cup F_n)$ for $n\in \N$. Suppose that $\liminf_{n\to\infty}\dist(x_n,\partial V_n)=0$. Since $F_n$ converges to $F$ as $n\to\infty$, we see that $\liminf_{n\to \infty}\dist(x_n,F)=0$, a contradiction. Thus, there exists $\varepsilon>0$ such that $\dist(x_n,\partial V_n)\geq \varepsilon$ for all $n\in \N$.  Since $V_n$ is a bounded component of $\C\setminus (F\cup F_n)$, we have $\diam (V_n)\leq \diam(F\cup F_n)$, which implies that $V_n$ lies in a fixed bounded neighborhood of $F$ for all $n\in \N$. By passing to a subsequence, we assume that $x_n$ converges to a point $x_0\in \C$. Then $B(x_0,\varepsilon/2)\subset V_n$ for all sufficiently large $n\in \N$. 

\begin{figure}
	\centering
	\begin{tikzpicture}
		\draw[line width=2pt, fill=black!10!white] (-0.6,0) circle (1.9cm);
		\draw[fill=white] (0,0) circle (2cm);
		\draw[line width=2pt] (-0.6,0) circle (1.9cm);
		\node at (-2.2,0) {$V_n$};
		\node at (-1.5,-2) {$F_n$};
		\node at (1,-2) {$F$};
	\end{tikzpicture}
	\caption{The assumptions of Lemma \ref{lemma:convergence}. The circles $F_n$ (in bold) converge to the circle $F$. The domain $V_n\subset \C\setminus (F\cup F_n)$ does not intersect the bounded component of $\C\setminus F$.} \label{figure:circles}
\end{figure}
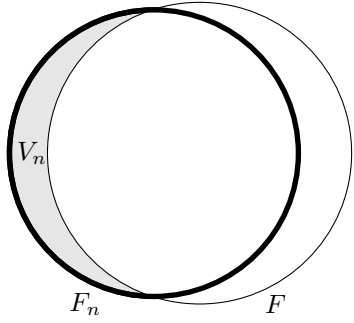

Let $V$ be the kernel of $\{V_n\}_{n\in\N}$ with base at $x_0$. That is, $V$ is the largest domain that contains $x_0$ with the property that each compact subset of $V$ lies in $V_n$ for all sufficiently large $n\in \N$. Note that $V$ is bounded.  Let $y\in \partial V$. Then for each $r>0$ we have $B(y,r)\not\subset V_n$ for infinitely many $n\in \N$; in particular $B(y,r)\cap \partial V_n\neq \emptyset$ for infinitely many $n\in \N$. This implies that $\liminf_{n\to\infty}\dist(y,\partial V_n)= 0$. Thus, $y\in F$.  We conclude that $\partial V\subset F$, so $\C\setminus \partial V\supset \C\setminus F$. The domain $V$ is a component of $\C\setminus \partial V$, so at least one component of $\C\setminus F$ is contained in $V$. If $W$ is a bounded component of $\C\setminus F$, then by assumption $W\cap V_n=\emptyset$ for all $n\in \N$, so $W\cap V=\emptyset$. We conclude that the unbounded component of $\C\setminus F$ is contained in $V$. This contradicts the boundedness of $V$.
\end{proof}

\begin{proof}[Proof of Theorem \ref{theorem:totally_disconnected_ned}]
Without loss of generality, $K$ has a non-degenerate connected component. It suffices to show that this component contains a totally disconnected compact set that is not $S$-removable. Hence, we may assume that $K$ is a non-degenerate continuum. 

Suppose first that $K$ contains a Jordan arc $\widetilde K$ (i.e., a homeomorphic copy of the interval $[0,1]$) and let $\Omega= \C\setminus \widetilde K$, so $\widetilde K=\partial \Omega$. Let $G$ be a countable dense subset of $\partial \Omega$. By Lemma \ref{lemma:countable}, for each $\varepsilon>0$ and each continuum $E\subset \Omega$ there exists an open set $U$ such that $G\subset U\subset N_{\varepsilon}(G)=N_{\varepsilon}(\partial \Omega)$, the set $K'=\partial \Omega\setminus U$ is totally disconnected, and 
$$\mod \Gamma(E,U;\C\setminus K')\leq \mod \Gamma(E,U;\C)<\varepsilon.$$
An application of Lemma \ref{lemma:totally_disconnected_ned_sufficient} completes the proof of Theorem \ref{theorem:totally_disconnected_ned} in that case.

Suppose, next, that the continuum $K$ contains no arcs. Let $\Omega$ be a component of ${\C}\setminus K$. Then  $\partial \Omega$ is a non-degenerate continuum that is a subset of $K$ and contains no arcs. For $n\in \N$, let $\Omega_n\subset \Omega$ be a neighborhood of $\partial \Omega$ in $\Omega$ with $\area(\Omega_n)<4^{-n}$. For $\varepsilon>0$ define 
$$\rho= c\varepsilon^{1/2} \sum_{n=1}^\infty n \chi_{\Omega_n},$$
where $c$ is a positive absolute constant so that  $c \sum_{n=1}^\infty n2^{-n}<1$. We have
$$\|\rho\|_{2}\leq c\varepsilon^{1/2} \sum_{n=1}^\infty n 2^{-n}<\varepsilon^{1/2}.$$
Let $E\subset  \Omega$ be a continuum.  We will find an open set $U\subset N_{\varepsilon}(\partial \Omega)$ such that the set $K'=\partial \Omega\setminus U$ is totally disconnected and 
$$\int_{\gamma}\rho \, ds \geq 1$$
for every curve $\gamma\in \Gamma(E,U;\C\setminus K')$. Assuming this, we have 
$$\mod \Gamma(E,U;\C\setminus K')\leq \|\rho\|_2^2< \varepsilon.$$
This shows that the assumptions of Lemma \ref{lemma:totally_disconnected_ned_sufficient} are satisfied, so Theorem \ref{theorem:totally_disconnected_ned} follows.  

We now focus on constructing the open set $U$. Let $G\subset \partial \Omega$ be a compact set. Consider two distinct parallel lines $L_1,L_2$ and the closed strip $W=W(L_1,L_2)$. Suppose that the compact set $W\cap G$ has infinitely many components intersecting both lines $L_1$ and $L_2$. Consider a sequence $F_n$, $n\in \N$, of pairwise disjoint such components that converges in the Hausdorff sense to a continuum $F\subset G \subset \partial \Omega$ intersecting both $L_1$ and $L_2$. Note that if $F_n\cap F\neq \emptyset$ for some $n\in \N$, then $F\subset F_n$; however, this can hold only for at most one index. We pass to a subsequence so that $F$ is disjoint from $F_n$ for all $n\in \N$.  Note that no bounded component of $W\setminus F$ intersects both $L_1$ and $L_2$, since $F$ is connected; thus, the sets $F_n$, $n\in \N$, lie in unbounded components of $W\setminus F$. The connectedness of $F$ implies that $W\setminus F$ has  precisely two unbounded components, so one of them contains infinitely many sets $F_n$, $n\in \N$. After passing to a subsequence, we assume that for all $n\in \N$ the set $F_n$ lies in the same unbounded component of $W\setminus F$; see Figure \ref{figure:ux}. 

For each $n\in \N$ there exist two line segments $S_{1,n}\subset L_1$ and $S_{2,n}\subset L_2$ that connect $F_n$ and $F$ and are disjoint from $F_n,F$, except at the endpoints. Since $F_n$ converges to $F$, we conclude that $S_{i,n}$ converges to a point of $F\cap L_i$ as $n\to\infty$ for $i=1,2$. Thus, $F_n'=F_n\cup S_{1,n}\cup S_{2,n}$ converges to $F$ in the Hausdorff sense. Consider the domain $V_n$ that is a bounded component of $\C\setminus (F\cup F_n')$ and $S_{1,n}\cup S_{2,n}\subset \partial V_n$. Note that $V_n$ does not intersect any bounded component of $\C\setminus F$. By Lemma \ref{lemma:convergence}, for each $\delta>0$ we have $V_n\subset N_\delta(F)$  for all sufficiently large $n\in \N$.

\begin{figure}
\centering
\begin{tikzpicture}
\draw (-5,2)--(7,2);
\draw (-5,-2)--(7,-2);

\draw[rounded corners=40, dashed] (4.5,-3)--(7.5,-3)--(7.5,3)--(4.5,3)--cycle;
\draw[line width=0.7mm] (4.8,-2)--(4.8,2) node[anchor=north west] {$F_n$};

\draw[fill=black!20] (6,-0.7)--(6,0.7) arc (90:270:0.7);
\node at (5.6,0) {$U_x$};
\draw[line width=0.7mm] (6,-2)--(6,2) node[anchor=north west] {$F$};
\fill (6,0) circle (2pt) node[anchor=west] {$x$};

\draw[rounded corners=30, densely dotted, line width=0.4mm] (0,-4)--(1,-3) -- (5,-4)--(5.4, -0.33);
\node at (1.5,-3) {$\gamma$};
\node at (7,-3.3) {$N_{\delta}(F)$};
\node[anchor=south] at (-4.5,-2) {$L_1$};
\node[anchor=south] at (-4.5,2) {$L_2$};

\draw (-3.5,-2)--(-3.5,2) node[anchor=north west] {$F_1$}; 
\draw (-2,-2)--(-2,2) node[anchor=north west] {$F_2$};
\draw (-1,-2)--(-1,2) node[anchor=north west] {$F_3$};
\draw (0,2)--(0,-2);
\draw (0.5,2)--(0.5,-2);
\draw (1,2)--(1,-2);
\draw (1.5,2)--(1.5,-2);
\draw (2,2)--(2,-2);
\draw (2.5,2)--(2.5,-2);
\draw (3,2)--(3,-2);
\draw (3.25,2)--(3.25,-2);
\draw (3.5,2)--(3.5,-2);
\draw (3.75,2)--(3.75,-2);
\draw (4.0,2)--(4.0,-2);
\draw (4.25,2)--(4.25,-2);
\end{tikzpicture}
\caption{Construction of the open set $U_x$.}\label{figure:ux}
\end{figure}
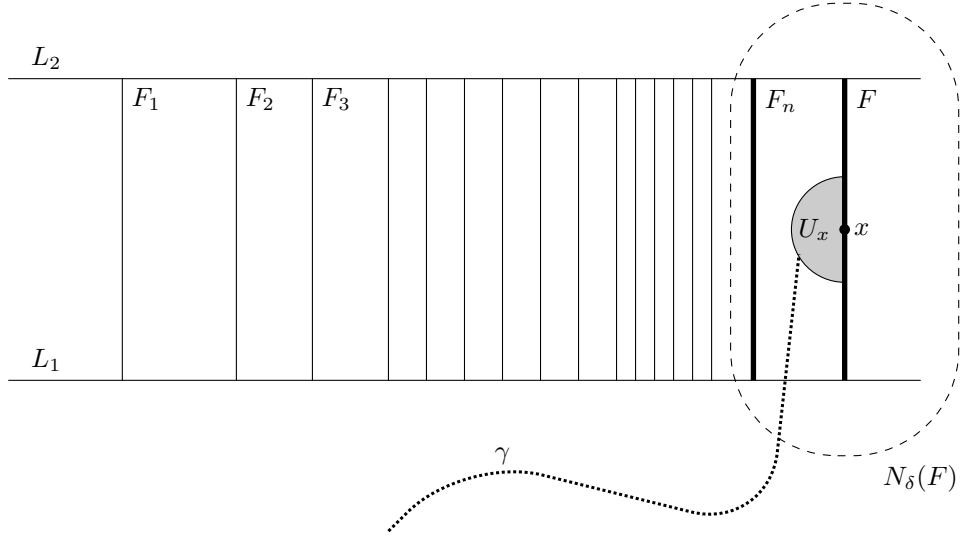

Let $x\in F\setminus (L_1\cup L_2)$, set $d_x=\dist(x,L_1\cup L_2)$, and consider a ball $B_x=B(x,r_x)$ with $r_x<d_x/2$. Let $\delta>0$ such that $\rho\geq 2 d_x^{-1}$ in $N_{\delta}(F)\cap \Omega$ and $E\cap N_{\delta}(F)=\emptyset$.  Let $n\in \N$ be such that $V_n\subset N_{\delta}(F)$. We define $U_x=B_x\cap V_n$; note that $U_x$ is not necessarily contained in $\Omega$. See Figure \ref{figure:ux} for an illustration.  If $\gamma\in \Gamma(E,U_x; \Omega)$, then $\gamma$ starts at $E$, enters the region $V_n$ at a point in $S_{1,n}\cup S_{2,n}\subset \partial V_n$, and has length at least $d_x/2$ in $V_n\subset N_{\delta}(F)$, until it intersects $U_x$. Therefore, 
$$\int_{\gamma}\rho\, ds \geq  1.$$
This shows that $\rho$ is admissible for $\Gamma(E,U_x;\Omega)$.

Now let $U=\bigcup_x U_x$, where the union is over all points $x\in F$,  all continua $F$ arising as limits of sequences $F_n$, $n\in \N$, as above,  all pairs of distinct parallel lines $L_1,L_2$, and all compact sets $G\subset \partial \Omega$. We set $K'=\partial \Omega\setminus U$. 

We show that $\rho$ is admissible for $\Gamma(E,U;\C\setminus K')$. By the above, $\rho$ is admissible for $\Gamma(E,U_x;\Omega)$ for each $x$, so it suffices to show that each path $\gamma\in \Gamma(E,U;\C\setminus K')$ has a subpath in $\Gamma(E,U_x;\Omega)$ for some $x$.  Let $\gamma\colon [a,b]\to \C$ be a curve in $\C\setminus K'$ connecting $E$ to $U$. That is, $\gamma(a)\in E$ and $\gamma(b)\in U_x$ for some $x$.  Let $[a,t)$ be the maximal interval such that $\gamma([a,t))\subset \Omega$. If $t=b$, then $\gamma\in \Gamma(E,U_x;\Omega)$, as desired. Suppose that $t<b$, so $\gamma(t)\in \partial \Omega$. In this case, since $\gamma(t)\notin K'$,  we necessarily have $\gamma(t)\in U$, so $\gamma(t)\in U_y$ for some $y$. Then the subpath $\gamma|_{[a,t]}$ of $\gamma$ lies in $\Gamma(E,U_y;\Omega)$.

Finally, we show that $K'=\partial \Omega\setminus U$ is totally disconnected. Consider a component $G$ of $\partial \Omega\setminus U$. Suppose that there exist distinct parallel lines $L_1,L_2$ such that $W(L_1,L_2)\cap G$ has infinitely many components intersecting both $L_1,L_2$. Then, as above we can find a sequence of distinct components $F_n$, $n\in \N$, of $W(L_1,L_2)\cap G$ that intersect both $L_1,L_2$ and converge to a continuum $F\subset G$. For each $x\in F\setminus (L_1\cup L_2)$ the open set $U_x$ intersects $F_n$ for large $n\in \N$ so $U\cap G\neq \emptyset$, a contradiction. Therefore, for any two parallel lines $L_1,L_2$ the set $W(L_1,L_2)\cap G$ has at most finitely many components intersecting both $L_1,L_2$. By Theorem \ref{theorem:characterization_locally_connected}, $G$ is a locally connected continuum. However, any two distinct points in a locally connected continuum can be connected by a Jordan arc \cite{Willard:topology}*{Theorem 31.2}. Since $\partial \Omega$ does not contain any Jordan arcs (as assumed in the beginning of the proof), we conclude that $G$ is a singleton. 
\end{proof}

\section{S-removable sets}
Let $\omega\colon \C\to [0,\infty)$ be a continuous function. For $x,y\in \C$, we define 
$$d_{\omega}(x,y)=\inf_{\gamma} \int_\gamma \omega\, ds,$$
where the infimum is taken over all rectifiable curves in $\C$ joining $x$ and $y$. 

\begin{theorem}\label{theorem:ned_hausdorff}
Let $E\subset \C$ be a compact set. The following are equivalent.
\begin{enumerate}[\upshape(i)]
	\item $E$ is $S$-removable.\label{nh:i}
	\item There exists a continuous function $\omega\colon \C\to [0,\infty)$ vanishing precisely on the set $E$ such that $\id\colon \C \to (\C, d_{\omega})$ is a quasiconformal homeomorphism and $\dim_{\h}(\id(E))=0$. \label{nh:ii}
	\item There exists a quasiconformal homeomorphism $f$ from $\C$ onto a metric surface $X$ such that $\h^1 (f(E))=0$. \label{nh:iii}
	\item For every metric surface $X$, each quasiconformal embedding $f\colon \C\setminus E \to X$ that extends continuously to $\infty$ with $f(\infty)=\infty$ extends to a quasiconformal homeomorphism from $\C$ onto $X$.\label{nh:iv}
	\item For every metric surface $X$, each homeomorphism $f\colon \C \to X$ that is quasiconformal on $\C\setminus E$ is quasiconformal on $\C$.\label{nh:v}
\end{enumerate}
\end{theorem}

In \ref{nh:iv} we are considering the one-point compactification of $X$ (recall that $X$ is homeomorphic to $\C$) and we are assuming that $f(\infty)=\infty$. Before giving the proof of Theorem \ref{theorem:ned_hausdorff} we discuss some preliminaries. 

\subsection{Quasiconformal maps}
Let $X,Y$ be metric surfaces. Recall that $X$ and $Y$ are assumed to be homeomorphic to $\C$ and to have locally finite Hausdorff $2$-measure. A homeomorphism $h\colon X\to Y$ is \textit{quasiconformal} if there exists $K\geq 1$ such that 
$$K^{-1} \mod\Gamma\leq \mod h(\Gamma)\leq K\mod\Gamma$$
for every curve family $\Gamma$ in $X$. A continuous map between topological spaces is \textit{cell-like} if the preimage of each point is a continuum that is contractible in each of its open neighborhoods; in the plane, such sets coincide with non-separating continua. A continuous, surjective, proper, and cell-like map $h\colon X\to Y$ is \textit{weakly quasiconformal} if  there exists $K>0$ such that for every curve family $\Gamma$ in $X$ we have
$$\mod \Gamma\leq K \mod h(\Gamma).$$
In this case, we say that $h$ is weakly $K$-quasiconformal. We note that continuous, proper, and cell-like maps from $X$ to $Y$ coincide with uniform limits of homeomorphisms; see \cite{Daverman:decompositions}*{Corollary 25.1A}.

\begin{lemma}[\cite{NtalampekosRomney:nonlength}*{Lemma 7.1}]\label{lemma:wqc_extend}
Let $X,Y$ be metric surfaces and $h\colon X\to Y$ be a continuous, surjective, proper, and cell-like map. Let $A\subset Y$ be a closed set with $\h^1(A)=0$ such that $h$ restricts to a weakly $K$-quasiconformal map from $X\setminus h^{-1}(A)$ onto $Y\setminus A$ for some $K>0$. Then $h$ is weakly $K$-quasiconformal.
\end{lemma}

A weakly $K$-quasiconformal map between planar domains is a $K$-quasiconformal homeomorphism. Indeed, by \cite{NtalampekosRomney:length}*{Theorem 7.4}, such a map is a homeomorphism. Also, note that a quasiconformal homeomorphism between planar domains is a priori required to satisfy only one modulus inequality, as in the definition of a weakly quasiconformal map; see \cite{LehtoVirtanen:quasiconformal}*{Section I.3}. 

Rajala \cite{Rajala:uniformization} introduced the notion of a \textit{reciprocal} metric surface, which we do not define here for the sake of brevity, and proved that a metric surface $X$ is quasiconformally equivalent to an open subset of $\C$ if and only if $X$ is reciprocal.

\begin{lemma}[\cite{MeierNtalampekos:rigidity}*{Lemma 2.13}]\label{lemma:wqc_qc}
Let $X$ be a metric surface and $h\colon X\to \C$ be a weakly quasiconformal map. Then $h$ is a quasiconformal homeomorphism.
\end{lemma}

\subsection{Conformal weights}
The next statement can be found \cite{NtalampekosRomney:length}*{Proposition 8.1} in the case that $\omega$ is the characteristic function of $\C\setminus E$. The proof can be easily adapted to give the current statement. 

\begin{lemma}\label{lemma:omega}
Let $E\subset \C$ be a totally disconnected compact set and  $\omega\colon \C\to [0,\infty)$ be a continuous function that vanishes precisely on $E$.  Then $(\C,d_{\omega})$ is a metric surface and the identity map $\id\colon \C\to (\C,d_\omega)$ is a homeomorphism that is $1$-quasiconformal on $\C\setminus E$. 
\end{lemma}

Ikonen and Romney \cite{IkonenRomney:removable} gave the following characterization of $S$-removable sets in connection with surfaces that arise by scaling the Euclidean metric with a degenerate conformal weight.

\begin{theorem}[\cite{IkonenRomney:removable}*{Theorems 1.3 and 1.4}]\label{theorem:ikonen_romney}
Let $E\subset \C$ be a totally disconnected compact set. Then $E$ is $S$-removable if and only if for every continuous function $\omega\colon \C\to [0,\infty)$  that vanishes precisely on $E$ the space $(\C,d_\omega)$ is a reciprocal metric surface.  
\end{theorem}

\begin{lemma}\label{lemma:dimension_zero}
Let $E\subset \C$ be a totally disconnected compact set. There exists a continuous function $\omega \colon \C\to [0,\infty)$ that vanishes precisely on $E$ such that the Hausdorff dimension of $E$ in the metric $d_{\omega}$ is zero.
\end{lemma}

The proof given below follows from an adaptation of the argument in \cite{IkonenRomney:removable}*{Lemma 5.1}. 

\begin{proof}
Let $\delta(x)=\dist(x,E)$ and define $\omega(x)=\min\{\delta(x)^{1/\delta(x)},1\}$ for $x\notin E$ and $\omega(x)=0$ for $x\in E$. Then $\omega$ is continuous and vanishes precisely on $E$, as required. We will show that the Hausdorff dimension of $E$ in the metric $d_\omega$ is zero. It suffices to show that for each $q>0$ we have $\h^q_{d_\omega}(E)=0$. Let $q>0$ and consider $p>1$ such that $(p+1)q>2$. Let $\varepsilon>0$. Since the Hausdorff dimension of $E$ is at most $2$ in the Euclidean metric and $(p+1)q>2$, we have $\h^{(p+1)q}(E)=0$. Thus, there exists a cover of $E$ by sets $\{A_j\}_{j\in \N}$ with $\diam A_j<1/p$ for each $j\in \N$ and
$$\sum_{j\in \N} (\diam A_j)^{(p+1)q}<\varepsilon.$$
Without loss of generality, $A_j\cap E\neq \emptyset$ for each $j\in \N$. Let $y_j\in A_j\cap E$, $d_j=\diam A_j$, and observe that for each $x\in \br B(y_j,d_j)$ we have 
$$\delta(x)=\dist(x,E)\leq |x-y_j|\leq d_j <\frac{1}{p}<1.$$
In particular, $\omega(x)\leq \delta(x)^{1/\delta(x)} <\delta(x)^p$. For each $x\in A_j$, by integrating over the straight line segment from $y_j$ to $x$, which is contained in $\br B(y_j,d_j)$, we have
$$d_\omega( y_j,x) \leq \int_{[y_j,x]}  \omega\, ds \leq \int_{[y_j,x]} \delta^p\, ds\leq \int_0^{d_j} t^p\, dt= (p+1)^{-1}d_j^{p+1}.$$
Thus, $\diam_{d_\omega}A_j\leq 2(p+1)^{-1}d_j^{p+1}$. We conclude that
$$ \sum_{j\in \N}(\diam_{d_\omega}A_j)^q \leq 2^q(p+1)^{-q} \sum_{j\in \N} d_j^{(p+1)q}<  2^q(p+1)^{-q}\varepsilon.$$
This shows that $\h^q_{d_\omega}(E)=0$ and completes the proof.
\end{proof}

\begin{proof}[Proof of Theorem \ref{theorem:ned_hausdorff}]
Suppose that $E$ is $S$-removable as in \ref{nh:i}, so in particular $E$ is totally disconnected. Consider a weight $\omega$ as in Lemma \ref{lemma:dimension_zero}, so that the Hausdorff dimension of $E$ in the metric $d_{\omega}$ is zero. Since $E$ is $S$-removable, Theorem \ref{theorem:ikonen_romney} implies that $(\C,d_{\omega})$ is reciprocal. By Rajala's theorem \cite{Rajala:uniformization}, there exists a quasiconformal embedding $\phi \colon (\C,d_\omega)\to \C$. By Lemma \ref{lemma:omega}, the identity map $\id\colon \C\to (\C,d_\omega)$ is a homeomorphism that is quasiconformal in $\C\setminus E$. Thus, $\phi\circ \id\colon \C\to \C$ is an embedding that is quasiconformal in $\C\setminus E$. Since $E$ is $S$-removable, every quasiconformal embedding of $\C\setminus E$ into $\C$ is the restriction of a quasiconformal map of $\widehat{\C}$ \cite{Younsi:removablesurvey}*{Proposition 4.7}. Thus, $\phi\circ \id$ is a quasiconformal homeomorphism of the plane and $\id\colon \C\to (\C,d_\omega)$ is quasiconformal. This completes the proof of \ref{nh:ii}. Also \ref{nh:ii} readily implies \ref{nh:iii}. 

We show that \ref{nh:iii} implies \ref{nh:iv}. Suppose that there exists a quasiconformal homeomorphism $f$ from $\C$ onto a metric surface $X$ such that $\h^1(f(E))=0$, as in \ref{nh:iii}. This implies that $f(E)$ is a totally disconnected compact set. Since $f$ is a homeomorphism, $E$ is also totally disconnected. Consider a quasiconformal embedding $g$ from $\C\setminus E$ into a metric surface $Y$ such that $g$ extends continuously to $\infty$ and $g(\infty)=\infty$, as in \ref{nh:iv}. Let $E'$ be the compact set $Y\setminus g(\C\setminus E)$, so from a topological point of view, $E'$ is the complement of a planar domain. In particular, each component of $E'$ is a non-separating continuum. The map $f\circ g^{-1}$ is a quasiconformal homeomorphism from $Y\setminus E'$ onto $X\setminus f(E)$. Since $f(E)$ is totally disconnected, the map $h=f\circ g^{-1}$ extends to a continuous, surjective, proper, and cell-like map from $Y$ onto $X$, which we also denote by $h$; see the discussion in \cite{NtalampekosYounsi:rigidity}*{Section 3.1}. Since $\h^1(f(E))=0$, Lemma \ref{lemma:wqc_extend} implies that $h$ is weakly quasiconformal. Therefore, $f^{-1}\circ h\colon Y\to \C$ is weakly quasiconformal. By Lemma \ref{lemma:wqc_qc}, $f^{-1}\circ h$ is a quasiconformal homeomorphism. Therefore, $g$ is the restriction of the quasiconformal homeomorphism $h^{-1}\circ f\colon \C \to Y$, as claimed in \ref{nh:iv}.

Suppose that \ref{nh:iv} is true. We show first that $E$ is totally disconnected. If not, let $F$ be a non-degenerate component of $E$. By the Riemann mapping theorem one can find a conformal embedding $f\colon \C\setminus F\to \C$ that fixes $\infty$ and does not extend to a homeomorphism of $\C$. This is a contradiction, so $E$ is totally disconnected. Now, consider a homeomorphism $f\colon \C\to X$ that is quasiconformal on $\C\setminus E$ as in \ref{nh:v}. Since $f$ is a homeomorphism, it extends to $\infty$ and $f(\infty)=\infty$. By \ref{nh:iv}, $f|_{\C\setminus E}$ has an extension to a quasiconformal homeomorphism $\widetilde f\colon \C\to X$. Since $E$ is totally disconnected, we must have $\widetilde f=f$. This shows \ref{nh:v}. 

Finally, we show that \ref{nh:v} implies \ref{nh:i}. Suppose that $E$ is not $S$-removable. By Theorem \ref{theorem:totally_disconnected_ned}, there exists a totally disconnected set $F\subset E$ that is not $S$-removable. By Theorem \ref{theorem:ikonen_romney}, there exists a continuous function $\omega \colon \C\to [0,\infty)$ vanishing precisely on $F$ such that the space $(\C,d_{\omega})$ is not reciprocal. Then by Lemma \ref{lemma:omega} the map $\id\colon \C\to (\C,d_{\omega})$ is a homeomorphism that quasiconformal on $\C\setminus F$. This map is not quasiconformal on $\C$, since this would imply that $(\C,d_\omega)$ is reciprocal.
\end{proof}

\section{CH-removable sets}\label{section:ch}

In this section we prove Theorem \ref{theorem:intro_ch}. We discuss some background on Sobolev spaces on metric spaces. See the monograph \cite{HeinonenKoskelaShanmugalingamTyson:Sobolev} for a detailed exposition. Let $\Omega$ be a metric measure space, $V$ be a Banach space, and $1\leq p< \infty$. The Newton--Sobolev space $N^{1,p}(\Omega,V)$ is defined as the space of measurable maps $u\in L^p (\Omega,V)$ with the property that there exists a Borel function $\rho\colon \Omega\to [0,\infty]$ with $\rho\in L^p(\Omega)$ and 
$$\|u(\gamma(a))-u(\gamma(b))\| \leq \int_{\gamma}  \rho \, ds$$
for all curves $\gamma\colon [a,b]\to \Omega$ outside a curve family $\Gamma_0$ with $\md_p\Gamma_0=0$. A function $\rho$ with these properties is called a $p$-weak upper gradient of $u$. If $u\in N^{1,p}(\Omega,V)$, then there exists a minimal $p$-weak upper gradient in $L^p(\Omega)$, which is denoted by $\rho_u$. We note that the minimal $p$-weak upper gradient is local in the sense that if two functions agree in an open set then their minimal weak upper gradients do too \cite{Williams:qc}*{Corollary 3.9}. The space $N^{1,p}_{\loc}(\Omega,V)$ is defined in the obvious manner.

\begin{lemma}\label{lemma:newton_removable}
Let $n\geq 1$, $\Omega\subset \R^n$ be an open set, $V$ be a Banach space, and $1\leq p< \infty$.
Let $f\colon \Omega\to V$ be a topological embedding. If $E$ is a closed subset of $\Omega$ such that $f(E)$ has $\sigma$-finite Hausdorff $1$-measure and $f|_{\Omega\setminus E}\in N^{1,p}(\Omega\setminus E,V)$, then $f\in N^{1,p}(\Omega,V)$ and $\rho_f= \rho_{f|_{\Omega\setminus E}}\chi_{\Omega\setminus E}$. 
\end{lemma}
\begin{proof}
Let $\rho=\rho_{f|_{\Omega\setminus E}}\chi_{\Omega\setminus E}$. Fix $j\in \{1,\dots,n\}$ and consider lines parallel to the direction $e_j$. For any two such lines $L_1,L_2$, the images $f(L_1\cap \Omega), f(L_2\cap \Omega)$ are disjoint, since $f$ is an embedding. The additivity of $\h^1$ on measurable subsets of $V$ and the assumption that $f(E)$ has $\sigma$-finite Hausdorff $1$-measure imply that $f(E\cap L)$ can have positive Hausdorff $1$-measure for at most countably many lines $L$ parallel to $e_j$. Therefore, almost every line $L$ parallel to $e_j$ has the property that $\h^1(f(E\cap L))=0$, $\int_{\gamma} \rho\, ds<\infty$ for every line segment $\gamma\colon [a,b]\to L\cap \Omega$, and
\begin{align}\label{lemma:newton_removable_ug}
\| f(\gamma(a))- f(\gamma(b))\|\leq \int_{\gamma} \rho \, ds
\end{align}
for every line segment $\gamma\colon [a,b]\to L\cap (\Omega\setminus E)$. Now, consider a line segment $\gamma\colon [a,b]\to L\cap \Omega$. Denote by $(a_i,b_i)$, $i\in I$, the interiors of the components of $\gamma^{-1}(L\cap (\Omega\setminus E))$ and set $\gamma_i=\gamma|_{(a_i,b_i)}$, $i\in I$. By the subadditivity of Hausdorff $1$-content and \eqref{lemma:newton_removable_ug}, we have
\begin{align*}
\| f(\gamma(a))-f(\gamma(b))\| &\leq \diam f(\gamma([a,b])) = \h^1_\infty(f(\gamma([a,b])))\\
&\leq \sum_{i\in I} \h^1_\infty( f(\gamma_i(a_i,b_i))) +  \h^1_\infty( f(E\cap L))
\\
&\leq \sum_{i\in I}\int_{\gamma_i} \rho\, ds \leq \int_\gamma \rho\, ds.
\end{align*}
This shows that $f$ is absolutely continuous in almost every line. 

For $z\in V$ we define $f_z(x)=\|f(x)-z\|$, $x\in \Omega$. By the above, $f_z$ is absolutely continuous in almost every line and $$\left|\frac{\partial f_z}{\partial x_j}\right| \leq  \rho $$
almost everywhere, for each $j\in \{1,\dots,n\}$. In particular, $|\nabla f_z|\leq \sqrt{n}\rho$. By \cite{Hajlasz:Sobolev}*{Theorem 7.13} (see also \cite{HeinonenKoskelaShanmugalingamTyson:Sobolev}*{Theorem 7.4.5}), we have $f_z\in N^{1,p}(\Omega,\R)$ with $|\nabla f_z|=\rho_{f_z}$. Hence, $\rho_{f_z}\leq \sqrt{n}\rho$ for each $z\in X$. By \cite{HeinonenKoskelaShanmugalingamTyson:Sobolev}*{Theorem 7.1.20}, $f\in N^{1,p}(\Omega,V)$ and $\rho_f\leq \sqrt{n}\rho$. In particular, $\rho_f=0$ on $E$. On the other hand, the locality of the minimal weak upper gradient implies that $\rho_f=\rho$ in the complement of $E$; see \cite{Williams:qc}*{Corollary 3.9}. 
\end{proof}

The notion of Newton--Sobolev space extends to functions whose target is a metric space as follows. Let $X$ be a metric space. By the Kuratowski embedding theorem we may assume that $X$ is isometrically embedded in a Banach space $V$. We define $N^{1,p}(\Omega,X)$ to be the space of functions in $N^{1,p}(\Omega,V)$ with values in $X$. Similarly one defines $N^{1,p}_{\loc}(\Omega,X)$. The next theorem of Williams \cite{Williams:qc}*{Theorem 1.1} gives the equivalence between the geometric and analytic definitions of quasiconformality.

\begin{theorem}[Definitions of quasiconformality]\label{theorem:qc_definitions_williams}
Let $X,Y$ be metric surfaces, $h\colon X\to Y$ be a continuous map, and $K>0$. The following are equivalent.
\begin{enumerate}[label=\normalfont(\roman*)]
    \item \label{item:qc_equivalence_i_williams} $h\in N^{1,2}_{\loc}(X,Y)$ and there exists a $2$-weak upper gradient $\rho$ of $h$ such that for every Borel set $E\subset Y$ we have $$\int_{h^{-1}(E)}\rho^2 \, d\h^2 \leq K \h^2(E).$$
    \item \label{item:qc_equivalence_ii_williams} For every curve family $\Gamma$ in $X$ we have
    $$\mod \Gamma \leq K \mod h(\Gamma).$$
\end{enumerate}
\end{theorem}

\begin{proof}[Proof of Theorem \ref{theorem:intro_ch}]
Let $f\colon \C\to X$ be a quasiconformal homeomorphism such that $f(E)$ has $\sigma$-finite Hausdorff $1$-measure. Let $g\colon \C\to \C$ be a homeomorphism that is quasiconformal in $\C\setminus E$. Our goal is to show that $g$ is quasiconformal in $\C$. Consider the map $h= f\circ g^{-1}\colon \C\to X$, which is a homeomorphism that is quasiconformal in $\C\setminus g(E)$. By Theorem \ref{theorem:qc_definitions_williams}, $h\in N^{1,2}_{\loc}(\C\setminus g(E), X)$, and there exists a $2$-weak upper gradient $\rho$ of $h|_{\C\setminus g(E)}$ such that for every Borel set $G\subset X\setminus f(E)$ we have $$\int_{h^{-1}(G)}\rho^2 \, d\h^2 \leq K \h^2(G).$$
Note that $h\in N^{1,2}(B\setminus g(E),X)$ for each ball $B\subset \C$. By Lemma \ref{lemma:newton_removable} we have $h\in N^{1,2}_{\loc}(\C, X)$ and $\rho_h=0$ on $g(E)$. If we set $\rho=0$ on $g(E)$, then $\rho_h\leq \rho$, so $\rho$ is a $2$-weak upper gradient of $h$ and we have
$$\int_{h^{-1}(G)}\rho^2 \, d\h^2 \leq K \h^2(G).$$
for every Borel set $G\subset X$. By Theorem \ref{theorem:qc_definitions_williams} we conclude that $h$ is a weakly quasiconformal map. Hence, $g^{-1}=f^{-1}\circ h$ is also weakly quasiconformal. Finally, note that weakly quasiconformal maps between planar open sets are quasiconformal; see Lemma \ref{lemma:wqc_qc} and the preceding discussion. We conclude that $g$ is a quasiconformal map, as desired.  
\end{proof}

\bibliography{biblio}

\end{document}